\documentclass[12pt,a4paper,oneside]{amsart}
\usepackage{color,amssymb,latexsym,amsfonts,textcomp,fancyhdr,calc,graphicx}
\usepackage{amsmath,amstext,amsthm,amssymb,amsxtra}
\usepackage[left=2.5cm,right=2.5cm,bottom=2cm,top=2cm]{geometry} 
\usepackage[colorlinks,citecolor=red,pagebackref,hypertexnames=false]{hyperref}
\usepackage{dsfont}
\usepackage[framemethod=tikz]{mdframed}
\usepackage[inline]{asymptote}

\usepackage{txfonts,pxfonts,tikz} 
\usepackage{tgschola}
\usepackage[T1]{fontenc}
\usepackage{mathtools}

\mathtoolsset{showonlyrefs}
\usepackage{arydshln}

\usepackage{nicematrix}

\usepackage{version}

\usepackage{booktabs} 

\makeatletter
\renewcommand*\env@matrix[1][*\c@MaxMatrixCols c]{%
  \hskip -\arraycolsep
  \let\@ifnextchar\new@ifnextchar
  \array{#1}}
\makeatother

\newtheorem{prevtheorem}{Theorem}

\newenvironment{Birkhoff}{{\it{Birkhoff's Theorem:} }}

\newenvironment{brualdi}{{\it{Theorem:} }}

\theoremstyle{plain}
\newtheorem{theorem}{Theorem}[section]
\newtheorem{lemma}[theorem]{Lemma}
\newtheorem{corollary}[theorem]{Corollary}
\newtheorem{proposition}[theorem]{Proposition}

\theoremstyle{definition}
\newtheorem{definition}[theorem]{Definition}

\newtheorem{case[theorem]}{Case}

\theoremstyle{remark}
\newtheorem{remark}{Remark}

\numberwithin{equation}{section}

\newcommand{\G}{{\mathcal{G}}}
\DeclareMathOperator{\F}{{\mathbb F}}
\newcommand{\Abs}[1]{{\left|{#1}\right|}}
\newcommand{\Set}[1]{{\left\{{#1}\right\}}}
\newcommand{\M}{\mathcal{M}}
\newcommand{\supp}{{\rm supp\,}}
\newcommand{\ZZ}{{\mathbb Z}}
\newcommand{\NN}{{\mathbb N}}

\newcommand{\bigzero}{\mathbf{0}}

\newenvironment{proofofB}{{\bf {Proof of Theorem \ref{Thm:B}.} }}{\hfill $\blacksquare$ \\}

\title{Doubly stochastic arrays with small support}
\author{Maria Loukaki}
\address{Department of Mathematics and Applied Mathematics, University of Crete, Voutes Campus, 70013 Heraklion, Crete, Greece.}
\email{mloukaki@uoc.gr}

\date{}

\begin{document}


\begin{abstract}
An $n  \times m$ non-negative array with row sum $m$ and column sum $n$ is called doubly stochastic.  We answer the problem of finding  doubly stochastic arrays  of smallest posible support for every $1 <n  \leq m$.  Any array of minimum support is extremal in the sence of convexity, while examples of extremal arrays that are not of minimum support are given. But when $n,m$ are coprime integers extremal arrays are precisely those of minimum support. 
\end{abstract}

\maketitle

\section{Introduction}

According to the definition given by Caron, et al.  in \cite{Caron} an $n \times  m$ array $A = (a_{i,j})$ with $a_{i,j} \geq 0$ is called  doubly stochastic  (with uniform marginals) if 
\[
\begin{aligned}
\sum_{i=1}^n a_{i,j}&=n &\text{for all } j=1, \cdots , m \\
\sum_{j=1}^m a_{i,j}&=m &\text{for all } i=1, \cdots , n\\
\end{aligned}
\]

The set of all  $n \times  m$ doubly stochastic arrays is denoted by $\M(n, m)$. Furthermore, two arrays in $\M(n,m)$ are called {\it equivalent} if one can be transformed into the other by permuting rows and columns.

We should mention  here that the above definition differs slightly from  the usual definition for  square 
doubly stochastic arrays (matrices). The common definition for $\M(n,n)$ requires the matrices  to have nonnegative entries and all row and column sums equal to 1. 
These matrices have been studied extensively,  see for example Chap. 2 in \cite{Mar}.

An array  $ M \in \M(n,m)$ is called extremal if it cannot be represented as a convex combination of other  doubly stochastic arrays different from $M$, that is, $M$ is an extremal element in the convex set $\M(n,m)$. 
For square $n \times n$ matrices, a full characterization of the extremal matrices  in $\M(n,n)$ is known by a classical result  due to G. Birkhoff \cite{Bir1}, that we state here, using the notation of  \cite{Caron} that we have adopted.

 \begin{Birkhoff}
 \it{$M \in \M(n,n)$ is extremal  if and only if $\frac{1}{n}M$ is a permutation matrix. That is,   $M \in \M(n,n)$ is extremal if and only if  $\frac{1}{n}M$ is equivalent to $I_n$,  the identity matrix.}
 \end{Birkhoff}

Several types of  characterization of  the extremal doubly stochastic arrays in $\M(n,m)$ exist using  either a matrix representation in some normal form,  
graph theory or faces of polyhedra, just  to mention a few. (The interested  reader  could look at the list presented in the introduction of \cite{Caron}). We point out here, that if $M \in \M(n,m) $ is extremal then all its entries are integers (see the first remarks in \cite{Caron}).

Li, et al, in \cite{Xin} have characterized extremal arrays  using their support, that is,  the set of their nonzero entries. In particular,  they proved that  a array $M \in \M(n,m)$ is extremal if and only if  its support $\supp(M)$ is unique in the set $\{ \supp(A) \mid  A \in \M(n,m)\}$ (Theorem 1  in \cite{Xin}).


In addition, the support of a doubly stochastic array has attracted the attention of  Kolountzakis and Papageorgiou \cite{Kol1} in relation with some  tiling problems.
If one views a $n \times m$ array as a function $f$ on the product of cyclic groups
$$
G = \ZZ_n \times \ZZ_m
$$
then, with the subgroups
$$
G_1 = \ZZ_n \times \Set{0},\ \ G_2 = \Set{0} \times \ZZ_m,
$$
the constant row sum and the constant column sum properties of the array are written as
\begin{equation}\label{tiling}
\sum_{g \in G_2} f(x-g) = m,\ \ \sum_{g \in G_1} f(x-g) = n,
\end{equation}
respectively, valid for all $x \in G$. In this language one seeks a nonnegative function $f$ on $G$, of as small a support as possible, which {\em tiles} simultaneously with the set of translates $G_1$ as well as $G_2$ (see \cite{Kol1} for a more precise definition).

These problems, of tiling simultaneously with various subgroups, derive \cite{kolountzakis1997multi} from a classical problem of Steinhaus who asked if there is a subset of the plane which tiles the plane simultaneously with all rotates of the lattice $\ZZ^2$. This problem is still very much open in case one asks for a measurable subset of the plane \cite{kolountzakis1999steinhaus} but the answer is known to be affirmative without the measurability requirement \cite{jackson2002sets}. Interestingly, in dimension 3 and higher the situation is the exact opposite: no measurable Steinhaus sets exist \cite{kolountzakis1999steinhaus,kolountzakis2002steinhaus} but we do not know if such sets exist if we drop measurability \cite{jackson2002sets}. In \cite{kolountzakis1999steinhaus} the problem was first investigated of how to find a \textit{function} $f$ (as opposed to indicator function for Steinhaus sets) on the plane which tiles simultaneously with a finite set of rotates of $\ZZ^2$ and whose support has small {\em diameter}. This problem was continued in \cite{Kol1} by examining the problem in a more general finite abelian group setting, the prototype of which is to ask for a function $f$ on $G$ satisfying \eqref{tiling} and has small support.

In \cite{Kol1} the  quantity     $S(n,m)$ was defined  as follows.

\begin{definition} 
$S(n,m) = \min\Set{\Abs{\supp A}:\ A \in \M(n,m)}.$

The arrays $M \in \M(n,m)$ with $\Abs{\supp M} = S(n,m)$ are called  {\bf minimum} arrays in $\M(n,m)$.
Furthermore, we call a column of $A \in \M(n,m)$ a {\bf monocolumn }  if it contains 
exactly one non-zero entry, which obviously should equal $n$. 
\end{definition}

It was shown (see Theorem 4.3 and Lemma 4.5 in \cite{Kol1})  that  $S(n, kn)= kn$ while  $S(n,kn+1)=(k+1)n$. In addition,  a question has been raised about the value of $S(n, kn+r)$ for $1<r <n$.
Our main theorem in this short  note  gives a complete answer to Question 7 in  \cite{Kol1} and  states the following.
 \begin{prevtheorem}
 For all   integers $1< n,  m $,  we have 
 $
S(n,m)=n+m- \gcd(n,m)
$
 \end{prevtheorem}

According to Corollary 2 in \cite{Xin} an array $A \in  \M(n,m)$ is not extremal if and only if  there exists $B \in \M(n,m)$ with $\supp B \subsetneq \supp A$. Hence every  minimum array in  $\M(n,m)$ is also  extremal.
This gives an easy  way to verify that an  array is extremal just by looking at the size of its support, if this happens to be  minimum. But  there are extremal arrays that are not minimum (some examples are given at the end of  this note) so the condition on the size of the support is only sufficient.  Nevertheless, when $n, m$ are coprime integers it is also necessary as the next  result states.
\begin{prevtheorem}\label{Thm:B}
Let $n,m$ be coprime integers. Then $M \in \M(n,m)$ is extremal if and only if $M$ is minimum. That is, $M$ is extremal if and only if \[\Abs{\supp M}= n+m-1
\]
\end{prevtheorem}

The rest of the paper contains a method to construct minimum arrays in $\M(n,m)$. In addition,  a family of examples of extremal arrays  whose size of support  is one more than the minimum is constructed.  Finally, a few more examples of arrays are given as 
counterexamples  to possible  generalizations.

\section{ Main Results }
We start with a  method to produce minimum  doubly stochastic arrays of  size $n \times m $ for all integers $1 \leq  n\leq  m $.

Is it already known (see Proposition 4 in \cite{Caron}),  that in the case $m=kn$ the array $E(n,kn) \in \M (n , kn)$  defined as

\begin{equation}\label{eq.matrix}
E(n,kn)=
 \left.\left( 
                  \vphantom{\begin{array}{c}1\\1\\1\\1\\1\\1\\1\\1\\ \end{array}}
                  \smash{\underbrace{
                      \begin{array}{ccc}
                      \overbrace{n \cdots  n}^k &\cdots                      & \cdots \\
                      \cdots                    & \overbrace{n  \cdots  n}^k & \cdots \\
                      \cdots                    & \cdots                     & \cdots\\
                       \cdots                   & \cdots                     & \overbrace{n  \cdots  n}^k
                      \end{array}
                      }_{ \, kn \text{\, columns}} }
              \right) \,  \right\}
              \,n \text{\, rows}
              \\
\\              
\end{equation}
\\
\\
is an extremal array of size $n \times kn$. Furthermore, $E(n,kn)$ is minimum since it has exactly one element per column. So, $|\supp(E(n,kn))|= kn=S(n,kn)$.

Assume now that $n, m \in \NN$ are given with $1< n \leq m$. We use the Euclidean algorithm applied to  $n,m$ to produce as many extremal arrays of type \eqref{eq.matrix} as the steps of the algorithm. That is, assume that the Euclidean  algorithm goes as follows:

\begin{equation}\label{eq1}
\begin{aligned}
m &=k_1n+r_1\\
n &=k_2r_1 +r_2\\
r_1&=k_3 r_2+r_3 \\
\ &\vdots\\
r_{t-2}&=k_tr_{t-1} + r_{t}\\
r_{t-1}&=k_{t+1}r_t
\end{aligned}
\end{equation}

Then at every step we produce the arrays 
$E(n, k_1n), E(r_1, k_2r_1),  \cdots , E(r_t, k_{t+1}r_t)$.
 We put them together in a block form to make an $n \times m$ array $\F(n,m)$ as follows
\begin{frame}
\footnotesize
\setlength{\arraycolsep}{0.4pt} 
\medmuskip = 1mu 
    \[
 \F(n,m)= \begin{pNiceMatrix}
&\Vdots  &  &   &E(r_1,k_2r_1)^T  &    &&   & &  \\
         &\dot  & \Ldots &  & & & && &  \\ 
            &    &  &. & & & & & & \\ 
 E(n, k_1n) &       &  &\Vdots & &   &   &   & & \\
         & & & &  &E(r_3, k_4r_3)^T & & & &\\
             &       &E(r_2, k_3r_2)     &\dot &\Ldots & && && \\ 
               &    &  & & &  &. & & & \\
             &       &   &    &\, \, \qquad \quad  \,E(r_4, k_5r_4)  &  &\Vdots &\ddots & &\\
            &       &    &     &  &  & && &B\\
 \end{pNiceMatrix}
 \]
 \end{frame}
 
 
 were $B= \begin{cases}
\begin{aligned}
          &E(r_t, k_{t+1}r_t) ^T &\quad   &\text{if $t$ is odd}\\
          &E(r_t, k_{t+1}r_t)    &\quad  &\text{if $t$ is even}\\
\end{aligned}
         \end{cases}$

 To clarify our method  we compute  $\F(8,27)$. The Euclidean Algorithm for $(8,27)$ is 
$$ 
 \begin{aligned}
     27 &=3 \cdot 8 +3\\
     8 &=2 \cdot 3 +2\\
     3 &=1 \cdot 2 +1\\
     2 &= 2\cdot 1
 \end{aligned}
 $$
Hence we form the arrays
\[
E(8, 3\cdot 8)=
\left.\left( 
                  \vphantom{\begin{array}{c}1\\1\\1\\1\\1\\1\\1\\1\\ \end{array}}
                  \smash{\underbrace{
                      \begin{array}{ccc}
                      \overbrace{8 \cdots  8}^3 &\cdots                      & \cdots \\
                      \cdots                    & \overbrace{8  \cdots  8}^3 & \cdots \\
                      \cdots                    & \cdots                     & \cdots\\
                       \cdots                   & \cdots                     & \overbrace{8  \cdots  8}^3
                      \end{array}
                      }_{24 }}
              \right)\quad  \right\}
              \,8
\]

\[ E(3, 2\cdot 3)^T= \begin{pmatrix}
 3 & & \\
 3 & & \\
   &3 & \\
   & 3 & \\
   & & 3\\
   & & 3\\
 \end{pmatrix}, \, \,  \, \quad \quad
E(2, 1 \cdot 2)=\begin{pmatrix}
2&\\
   &2\\
 \end{pmatrix}, \, \, \, 
 \qquad
 E(1, 2\cdot 1)^T=\begin{pmatrix}
   1\\
   1\\
 \end{pmatrix}
\]

Putting  them together we  get 

\begin{center}
 \begin{align*}
\F&(8,27)= \\
\ &    \left(
    \begin{array}{@{}*{28}{c}@{}}
&8 &8 &8 &  &  &  & & & & & & & & & & & & & & & & & &     &3 & & \\
&  &  &  &8 &8 &8 & & & & & & & & & & & & & & & & & &      &3 & &\\
&  &  &  &  &  &  &\ldots & & & & & & & & & & & & & & & & &  & &3 &\\
&  &  &  & & & & & & &\ldots & & & & & & & & & & & & & &     & &3 &\\
&  &  &  & & & & & & & & & &\ldots & & & & & & & & & & &      & & &3\\
&  &  &  & & & & & & & & & & & & &\ldots & & & & & & & &     & & &3\\
&  &  &  & & & & & & & & & & & & & & & &8 &8 &8 & & &    &2 & &1\\
&  &  &  & & & & & & & & & & & & & & & & & & &8 &8 &8     & &2 &1\\
\end{array}
\right).
\end{align*}
\end{center}

The following remark is a special case of Proposition 2 in \cite{Caron},  according to which an array $A \in \M(n,m)$ is extremal if and only if there is no "cycle" in its support. 

\begin{remark}\label{Rem:1}
If $A \in \M(n,m)$ is extremal then $A$ does not contain a "square" of non-zero entries. That is,  there does not exist  non-zero entries $x,y,z,w $  in $A$  that form a square: 
\[
A=\begin{pmatrix}
  &  &  &  & &\\
   & &x  &  &y &\\
  &  & &  &    & \\
    &  &z  & &w &\\
   & & & & & \\
\end{pmatrix}
\]

\end{remark}

In order to prove Theorem 1 we need the following 
\begin{proposition}\label{Prop:mono}
Let $m = kn +r$ with $0< r <  n$. Then there exists a minimum array $A \in \M(n,m)$ with exactly $kn$  monocolumns.  In other words, there exists a minimum array $A$ so that   every row of $A$ has exactly $k$ entries equal to $n$.
\end{proposition}
 
\begin{proof}

Assume proposition does not hold for $\M(n,m)$. Let $p$ be the maximum  number of monocolumns a minimum array in $\M(n,m)$ can have, hence $p < kn$. We define
\[
X:= \{ A \in \M(n,m) \mid A \text{\, is minimum and has exactly $p$ monocolumns } \}. 
\]
Clearly every $A \in X$  contains at least one row that does not have $k$ entries equal to $n$ (or else $p=kn$).
Among the entries  of those  rows (the rows that contain less than $k$ entries equal to $n$) we write $m(A)$ for the maximum entry strictly less than $n$. Let 
\[
x:= \max \{ m(A)  \mid A \in X\}
\]
and assume $A_0 \in X$ is such that $m(A_0)= x$. Clearly $0< x < n$ and we assume that $x$ is in the $(i_0, j_0)$ entry  of the matrix $A_0$. Then looking at the $j_0$ column of $A_0$ we deduce that there exist positive integers $t_1, \cdots , t_l$ in the $j_0$ column of $A_0$ apart from $x$ such that  $0< t_1, \cdots, t_l <n$  while 
\[
\sum_{i=1}^l t_i + x = n
\]
Similarly looking at the $i_0$ row of $A_0$ we conclude that there exist positive integers $s_1, \cdots s_u$ in the $i_0$-row of $A_0$ apart from $x$  such that $0 < s_1, \cdots , s_u < n$ while
\[
\sum_{j=1}^u s_j + x \geq  n+r
\]
where the last inequality follows from the fact that at the $i_0$-row of $A_0$ exist less than $k$ entries equal to $n$, while the sum of all  the elements of the row
equals $m = kn +r$. 
Clearly $l , u >0$ while   $\sum t_i < \sum s_j$.

{\bf Case 1} Assume first that  $t _1 \leq s_j$ for some $j= 1, \cdots , u$. 

If $t_1$ is in $(a_1, j_0)$ position   of $A_0$ and  $s_j$ is in the $(i_0,b_j) $ one,  observe that the entry in the  $(a_1, b_j)$ position  of $A_0$ equals $0$ because otherwise  the entries 
\[
(a_1, j_0), (i_0, j_0), (i_0, b_j), (a_1, b_j)
\]
form a non-zero "square" in $A_0$ contradicting Remark \ref{Rem:1} (as $A_0$ is minimum).

Now we construct an array  $B$ from $A_0$ in the following way:
Every entry of   $B$ is identical with the corresponding entry of $A$ apart from the four entries lying in the positions
$(a_1, j_0), (i_0, j_0), (i_0, b_j), (a_1, b_j)$. In those  positions the entries of $A_0$ were $t_1, x, s_j, 0$ (in the order they appear)  and we replace them with the entries $0, x+t_1, s_j-t_1, t_1$ respectively.  That is 
\[
A_0=\begin{pmatrix}
  &  &  &  & &\\
   & &t_1  &  &0 &\\
  &  & &  &    & \\
    &  &x  & &s_j &\\
   & & & & & \\
\end{pmatrix}
\longrightarrow
B =   \begin{pmatrix}
  &  &  &  & &\\
   & &0  &  &t_1 &\\
  &  & &  &    & \\
    &  &x+t_1  & &s_j-t_1 &\\
   & & & & & \\
\end{pmatrix}.
\]
Clearly  $B \in \M(n,m)$ (the row and column sums have remained unchanged). Furthermore,  $|\supp B | \leq |\supp A_0 | $
 and thus they are equal as $A_0$ is minimum. Hence $B $ is also minimum. 
In addition the monocolumns of $A_0$ have been transferred unchanged to  monocolumns of $B$ (as $s_j < n$). Hence the number of monocolumns of $B$ can't be less than the number of monocolumns of $A_0$ and thus it is exactly $p$ ($p$ being maximum). We conclude that  $x+t_1 < n$  while $B \in X$. Now,  the $i_0$-row of $B$ has less than $k$ entries equal to  $n$ 
(actually it is the same number as the one in the $i_0$-row of $A_0$) 
and in position $(i_0, j_0)$ its entry is $x+t_1> x$. Hence  
\[
m(B) \geq x+t_1 > x=  \max \{ m(A)  \mid A \in X\}.
\]
This final contradiction finishes  Case 1.

{\bf Case 2} Assume now that $t_1 > s_j$  for all  $j= 1, \cdots , u$. 

Assume again that  $t_1$ is in $(a_1, j_0)$ position   of $A_0$ while $s_1$ is in $(i_0,b_1) $ and  observe (as in Case 1) that the entry in the  $(a_1, b_1)$ position  of $A_0$ equals $0$. 

Now we construct an array  $B$ from $A_0$ in a similar way as in Case 1. That is,  every entry of   $B$ is identical with the corresponding entry of $A$ apart from the four entries lying in the positions
$(a_1, j_0), (i_0, j_0), (i_0, b_1), (a_1, b_1)$. In those  positions the entries of $A_0$ were $t_1, x, s_1, 0$ (in the order they appear)  and we replace them with the entries $t_1-s_1, x+s_1, 0, s_1$ respectively.  That is 
\[
A_0=\begin{pmatrix}
  &  &  &  & &\\
   & &t_1  &  &0 &\\
  &  & &  &    & \\
    &  &x  & &s_1 &\\
   & & & & & \\
\end{pmatrix}
\longrightarrow
B =   \begin{pmatrix}
  &  &  &  & &\\
   & &t_1-s_1  &  &s_1 &\\
  &  & &  &    & \\
    &  &x+s_1  & &0 &\\
   & & & & & \\
\end{pmatrix}
\]
A similar argument as in Case 1 implies that $B \in X$ while 
\[
m(B) \geq x+s_1 > x=  \max \{ m(A)  \mid A \in X\}.
\]
This final contradiction finishes  Case 2 and completes the proof of the proposition.
\end{proof}

 An immediate consequence of Proposition \ref{Prop:mono} are the following two corollaries.
 \begin{corollary}\label{Cor:1}
 Let $m = kn +r$ with  $0< r <n $. Then there exists a minimum array in $\M(n, m)$ that is equivalent to 
 \[
 B=  \begin{pmatrix}
& E(n,k n)  &| & \hat{B}^T  &\\
\end{pmatrix}
 \]
 where $\hat{B} \in \M(r, n)$. 
 \end{corollary}

\begin{corollary}\label{Cor:2}
 Let $m = kn +r$ with  $0< r <n $. Assume $A= \begin{pmatrix} & E(n,k n)  &| & \hat{A}^T  &\\
\end{pmatrix} $ where $\hat{A} \in M(r, n) $. If $\hat{A}$ is minimum in $\M(r,n)$ then $A$ is minimum in $M(n,m)$.
 \end{corollary}

\begin{proof}
According to Corollary \ref{Cor:1} there exists a minimum array  $B \in \M(n,m)$ so that  $B = \begin{pmatrix}
& E(n,k n)  &| & \hat{B}^T  &\\
\end{pmatrix}$ with $\hat{B} \in \M(r, n)$. Hence 
$|\supp \hat{A} | \leq |\supp \hat{B} |$ as $\hat{A}$ is minimum in $\M(r,n)$. Hence 
\[
| \supp B| = kn + |\supp \hat{B} | \geq kn + |\supp \hat{A} | = |\supp A|
\]
and thus $|\supp A|=|\supp B| $ and $A$ is minimum.
\end{proof}

We are ready now to prove Theorem I  that we restate using the arrays $\F(n,m)$.
\begin{lemma}
The arrays $\F(n,m)$ are minimum and thus extremal in $\M(n,m)$. In addition,  
$$
S(n,m)=\Abs{\supp \F(n,m)}=n+m-\gcd(n,m).
$$
\end{lemma}

\begin{proof}
As we have already observed, every minimum array is also extremal.   
To show that    $\F(n,m)$ is minimum we induct on the number of steps needed to complete the Euclidean Algorithm. Note that in view of our notation above, this number is $t+1$. 
If $t=0$, that is  $m=kn$,   the array $E(n, m)$  is of minimum support. So our induction begins.

For the inductive step observe that if the Euclidean algorithm starts with 
$m = k_1n +r_1$ our construction guarantees that  $\F(n,m)$ is the sum of the following two arrays, whose blocks are associated with the same column partition 
\[
A=  \begin{pmatrix}
& E(n,k_1n)  &| & \bigzero  &\\
\end{pmatrix}
\text{ and } 
B= \begin{pmatrix}
&  \bigzero &| &\F(r_1,n)^T & \\
\end{pmatrix}
\]
As $\F(n,m)= A+B$  its                                                                                                                                                                                                                   first $k_1n$ columns are monocolumns, exactly those of $E(n, k_1n)$. Hence according to Corollary \ref{Cor:2}  the array  
\[\F(n,m)= \begin{pmatrix}
& E(n,k_1n)  &| &\F(r_1,n)^T & \\
\end{pmatrix} 
\] is minimum  if $\F(r_1,n)$ is minimum in $\M(r_1,n)$. 
The steps needed  in the  Euclidean algorithm for $(r_1,n)$ are  one  less than those needed for the pair $(n,m )$. Hence the inductive hypothesis implies that    $\F(r_1,n)$ is minimum in  $\M(r_1,n)$. Therefore, $\F(n,m)$ is minimum in $\M(n,m)$ and 
$S(n,m)=\Abs{\supp \F(n,m)}$. 

To compute $\Abs{\supp \F(n,m)}$ we note that in view of \eqref{eq1} and the way $\F(n,m)$ is constructed we get 

\begin{align*}
 \Abs{\supp \F(n,m)} &=  \Abs{\supp E(n,k_1n)} +\Abs{\supp E(r_1,k_2r_1)} + \cdots + \Abs{\supp E(r_t,k_{t+1}r_t)}\\    
  &= k_1n+k_2r_1+\cdots +k_{t}r_{t-1} +k_{t+1}r_t \\
  &= m-r_1+n -r_2 +\cdots + r_{t-2}-r_t+ r_{t-1} \\ 
  &= m+n-r_t
\end{align*}

But  the last non zero remainder in the Euclidean Algorithm (that is $r_t$) is the greatest common divisor of $(n,m)$. This completes the proof of the theorem. 
\end{proof}

According to Proposition 4 in \cite{Caron}, when $r=0$, that is $m = kn$, all the extremal arrays in $\M(n, kn)$ are equivalent to $E(n, kn)$ and thus are minimum. Hence, minimum and extremal arrays coincide in  $\M(n, kn)$.

Furthermore,  according to  Proposition 6 of \cite{Caron}, the same holds when $r=1$. That is, if $m= kn+1$,  every extremal  array $M$  in $ \M(n,m)$ satisfies  
$$\Abs{\supp M }= (k+1)n = n+m -1= S(n, m)$$
and thus  $M$ is extremal if and only if is   minimum. 

This neat characterization of extremal arrays does not hold for $r >1$ in general. A counterexample is given by  the extremal $4 \times 6$ array   
\begin{equation}\label{matrix:C}
T=\begin{pmatrix}
2 &2 &  &  &2  &\\
2 &  &4 &  &  & \\
  &2 &  &4 &  & \\
  &  &  &  &2 &4\\
\end{pmatrix}
\end{equation}
whose support contains 9 non zero entries while  $S(4,6)=8$. One can check that the array is extremal using, for example, Proposition 2 in \cite{Caron}.

Nevertheless,  when $\gcd(n,m)=1$, extremal and minimum arrays in $ \M(n,m)$ coincide. This is  our Theorem II,  that we are  now ready to prove.

\begin{proofofB}
In view of Theorem 5 in \cite{Caron} every extremal array $M \in \M(n,m) $ (with $m = kn+r$ ) is equivalent to the sum of two arrays $M_B$ and $M_R$, were every row of $M_B$ has exactly $k+1$ positive entries while $M_R$ has at  most $r-1$ positive entries. 
Hence every  extremal array $M \in \M(n,m) $ satisfies 
\[
(k+1)n \leq \Abs{\supp M } \leq (k+1)n +(r-1).
\]
On the other hand, if $\gcd(n,m)=1$  and  $m=kn+r$  we get 
\[S(n,m)= n+m -1 = (k+1)n +(r-1).
\]
We conclude that for every extremal $M$ we have 
\[S(n,m) \leq \Abs{\supp M } \leq (k+1)n +(r-1)=S(n, m)
\]
Hence $ \Abs{\supp M } =n+m-1$ and the proposition follows.
\end{proofofB}

The array $T$ in \eqref{matrix:C} is not the only example of an extremal array that is not minimum, but it is of smallest dimensions. Actually, we can produce  arbitrarily large  extremal non-minimum arrays as the next result states.

\begin{prevtheorem}
For every pair of integers $n,m$   that satisfy   
\begin{equation}\label{eq:nmd}
     m=k_1n+d  \text{\,  where  \, } n> d >1 \text{ \, and \, } n=k_2d, 
\end{equation}
there exist an  extremal  array in $\M(n,m)$ that is not minimum. 
\end{prevtheorem}

For its proof we will use a  characterization of extremal  arrays using their associated graphs given by Brualdi \cite{Brou1}. We first define the associated graph $\G(A)$ of any $n \times m$ array  $A=(a_{i,j})$ with $a_{i,j} \geq 0$   as follows. For every row $i$ and every column $j$  we get a node $x_i$ and $y_j$ respectively, for $1 \leq i \leq n$ and $1 \leq j \leq m$. There is an edge joining  $x_i$ and $y_j$ if and only if $a_{i,j} > 0$. Then the following theorem  holds, see  \cite{Brou1} and \cite{Brou2}.
           
             \begin{brualdi}
             \it{ A matrix  $M \in \M(n,m)$ is extremal if and only if  the connected components of $\G(M)$ are trees. Equivalently, $\G(M)$ has no cycles.}
             \end{brualdi}

 We are now ready to prove Theorem III.   
            
\begin{proof}
Assume $n,m$ are as above then $\gcd(n,m)=d$ while $S(n,m)=n+m-d$. The Euclidean Algorithm stops in two steps and our method produces  

\[
\F(n,m)= \begin{pNiceMatrix}
E(n,k_1n) &\mid  &E(d,k_2d)^T \\
\end{pNiceMatrix}
\]

which is equivalent to the following array in block form

 \[
 X=
 \underbrace{ 
\begin{pNiceMatrix}
B &       & \\
  &\Ddots & \\
 &        &B \\
\end{pNiceMatrix}  
}_{d \text{\, $B$-blocks }}  
\]
where every block $B$  is 

\[ 
B=\small
\left. \left( 
                  \vphantom{\begin{array}{c}1\\1\\1\\1\\1\\1\\1\\ 1\\ \end{array}}
                  \smash{\underbrace{
                      \begin{array}{ccccc}
                      &d & \overbrace{n \cdots  n}^{k_1} &                       &  \\
                       &d  &                    & \overbrace{n  \cdots  n}^{k_1} & \\
                      &\vdots  &                   &\ddots                    & \\
                       &d &                   &                      &  \overbrace{n  \cdots  n}^{k_1}
                       \end{array}
                      }_{k_1k_2+1 \text{ columns } } }
              \right) \quad \right\} 
              \,k_2 \text{ rows}
            \\
 \]
              
 \ \\
 \ \\

As $d \geq 2$ there exist at least two blocks in the array $X$. We replace  the first two $B$-blocks in $X$  with the  $2k_2 \times 2(k_1k_2 +1)$ array

\[
\small
\small
 C = \left(   \phantom{\begin{matrix}\\ \quad \\  0 \\ \quad \\ \quad \\ \quad \\  \overbrace{n \cdots  n}^{k_1}\\  \ddots\\  \overbrace{n  \cdots  n}^{k_1} \\\end{matrix}}
\hspace{-3em}
                  \begin{array}{@{}*{11}{c}@{}}
                      &d      &\quad n-d \,  &\quad \overbrace{n \cdots  n}^{k_1-1} &&                                &\quad                                  &d &\quad  & &\\
                      &d      &\quad0    &                                & \overbrace{n \cdots  n}^{k_1} &                                &\quad                                 &0 & & &\\
                      &\vdots &     &                                &\ddots &                                &                                 &  &\quad  & & \\
                       &d     &     &                                 && \overbrace{n  \cdots  n}^{k_1} &                                 &  &\quad  & & \\
                       &      &\quad d    &                                 & &                               &\overbrace{n  \cdots  n}^{k_1}  &\quad  & &&\\
                       &      &     &                                 & &                               &                    &\quad d &\quad  \overbrace{n \cdots  n}^{k_1} & &  \\
                       &      &     &                                 & &                               &                    &\quad \vdots &\quad                        & \ddots & \\
                       &      &     &                                 & &                               &                    &\quad d &\quad &&  \overbrace{n  \cdots  n}^{k_1}\\
                 \end{array} 
                 \right)
 \hspace{-2em}
\phantom{\begin{matrix}\\\quad \\  0 \\ \quad \\ \quad \\ \quad \\  \overbrace{n \cdots  n}^{k_1}\\  \ddots\\  \overbrace{n  \cdots  n}^{k_1} \\ \end{matrix}}
\begin{array}{l}
\left.\lefteqn{\phantom{\begin{matrix} \\ \quad \\  0 \\ \quad \\ \quad \\ \quad  \\ \overbrace{n \cdots  n}^{k_1} \end{matrix}}}\right\}k_2+1 \\
\left.\lefteqn{\phantom{\begin{matrix} \\ \overbrace{n \cdots  n}^{k_1}\\  \ddots\\  \overbrace{n  \cdots  n}^{k_1}  \end{matrix}}} \right\}k_2-1 
\end{array}
 \]

So we get 
  \[
  Y=\underbrace{
\begin{pNiceMatrix}
C& &       & \\
&B & & \\
&  &\Ddots & \\
 & &        &B \\
\end{pNiceMatrix}  
}_{d-2 \text{\, $B$-blocks }}  
\]
Observe that the array $T$ in \eqref{matrix:C} is a special case  of $Y$ when $d=2=k_2$ and $ k_1=1$.      
   
    Clearly $Y$ is not minimum as 
    \[
    \Abs{\supp C}= k_1+2 + (k_1+1)\cdot (2k_2-1)= 2k_2(k_1+1)+1 = 2\Abs{\supp B}+1
    \]
and therefore 
    \[
    \Abs{\supp Y}=\Abs{\supp C}+ (d-2) \Abs{\supp B}= d\Abs{\supp B} +1 = \Abs{\supp X}+1
    \]
    
It remains to show that  $Y$ is extremal. 
              $Y$ is defined as a direct  sum of the  block arrays $C$ and $(d-2)$-copies of $B$. Each one of those blocks contributes to the graph $\G(Y)$ one or more  connected components. Clearly those components that are associated with $B$ are trees. (This can be seen either directly from the array $B$ or from the fact that $X$ is extremal and $X$ is a direct sum of $d$ blocks, all equal to $B$.)
              
              We conclude that $Y$ is extremal if and only if the associated graph $\G(C)$ of $C$ is a tree. Which is indeed so, as the graph  $\G(C)$ is
               
\begin{center}     
        \,   \begin{picture}(0,0)%
\includegraphics{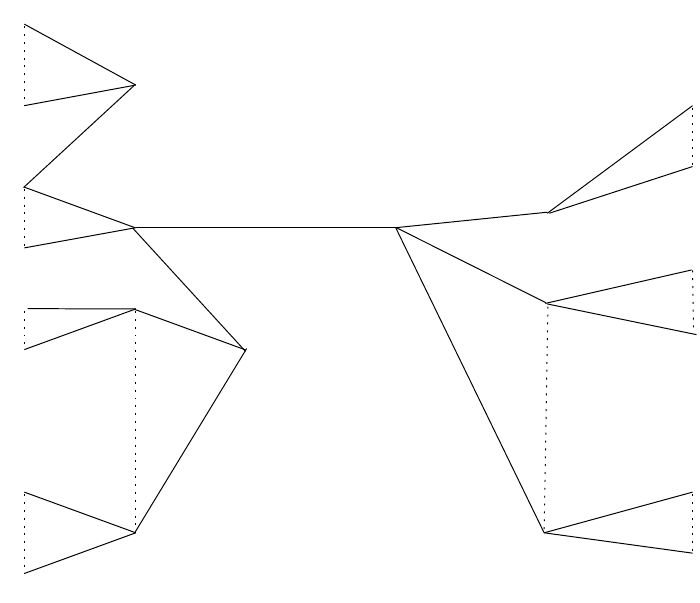}%
\end{picture}%
\setlength{\unitlength}{4144sp}%
\begingroup\makeatletter\ifx\SetFigFont\undefined%
\gdef\SetFigFont#1#2#3#4#5{%
  \reset@font\fontsize{#1}{#2pt}%
  \fontfamily{#3}\fontseries{#4}\fontshape{#5}%
  \selectfont}%
\fi\endgroup%
\begin{picture}(5319,4575)(111,-6213)
\put(5106,-5375){\makebox(0,0)[lb]{\smash{{\SetFigFont{8}{9.6}{\rmdefault}{\mddefault}{\updefault}{\color[rgb]{0,0,0}$y_{k_1(2k_2-1)+3}$}%
}}}}
\put(1145,-2430){\makebox(0,0)[lb]{\smash{{\SetFigFont{8}{9.6}{\rmdefault}{\mddefault}{\updefault}{\color[rgb]{0,0,0}$x_{k_2+1}$}%
}}}}
\put(1145,-5840){\makebox(0,0)[lb]{\smash{{\SetFigFont{8}{9.6}{\rmdefault}{\mddefault}{\updefault}{\color[rgb]{0,0,0}$x_{k_2}$}%
}}}}
\put(4257,-3360){\makebox(0,0)[lb]{\smash{{\SetFigFont{8}{9.6}{\rmdefault}{\mddefault}{\updefault}{\color[rgb]{0,0,0}$x_{k_2 +2}$}%
}}}}
\put(4257,-5840){\makebox(0,0)[lb]{\smash{{\SetFigFont{8}{9.6}{\rmdefault}{\mddefault}{\updefault}{\color[rgb]{0,0,0}$x_{2k_2}$}%
}}}}
\put(296,-4444){\makebox(0,0)[lb]{\smash{{\SetFigFont{8}{9.6}{\rmdefault}{\mddefault}{\updefault}{\color[rgb]{0,0,0}$y_{2k_1+1}$}%
}}}}
\put(296,-6151){\makebox(0,0)[lb]{\smash{{\SetFigFont{8}{9.6}{\rmdefault}{\mddefault}{\updefault}{\color[rgb]{0,0,0}$y_{k_1k_2+1}$}%
}}}}
\put(5106,-3049){\makebox(0,0)[lb]{\smash{{\SetFigFont{8}{9.6}{\rmdefault}{\mddefault}{\updefault}{\color[rgb]{0,0,0}$y_{k_1(k_2+2)+2}$}%
}}}}
\put(5106,-2430){\makebox(0,0)[lb]{\smash{{\SetFigFont{8}{9.6}{\rmdefault}{\mddefault}{\updefault}{\color[rgb]{0,0,0}$y_{k_1(k_2+1)+3}$}%
}}}}
\put(5106,-3670){\makebox(0,0)[lb]{\smash{{\SetFigFont{8}{9.6}{\rmdefault}{\mddefault}{\updefault}{\color[rgb]{0,0,0}$y_{k_1(k_2+2)+3}$}%
}}}}
\put(5106,-4289){\makebox(0,0)[lb]{\smash{{\SetFigFont{8}{9.6}{\rmdefault}{\mddefault}{\updefault}{\color[rgb]{0,0,0}$y_{k_1(k_2+3)+2}$}%
}}}}
\put(5106,-5995){\makebox(0,0)[lb]{\smash{{\SetFigFont{8}{9.6}{\rmdefault}{\mddefault}{\updefault}{\color[rgb]{0,0,0}$y_{2k_1k_2+2}$}%
}}}}
\put(4317,-4079){\makebox(0,0)[lb]{\smash{{\SetFigFont{8}{9.6}{\rmdefault}{\mddefault}{\updefault}{\color[rgb]{0,0,0}$x_{k_2+3}$}%
}}}}
\put(1120,-3909){\makebox(0,0)[lb]{\smash{{\SetFigFont{8}{9.6}{\rmdefault}{\mddefault}{\updefault}{\color[rgb]{0,0,0}$x_2$}%
}}}}
\put(291,-5305){\makebox(0,0)[lb]{\smash{{\SetFigFont{8}{9.6}{\rmdefault}{\mddefault}{\updefault}{\color[rgb]{0,0,0}$y_{k_1(k_2-1)+2}$}%
}}}}
\put(2029,-4344){\makebox(0,0)[lb]{\smash{{\SetFigFont{8}{9.6}{\rmdefault}{\mddefault}{\updefault}{\color[rgb]{0,0,0}$y_1$}%
}}}}
\put(2794,-3236){\makebox(0,0)[lb]{\smash{{\SetFigFont{8}{9.6}{\rmdefault}{\mddefault}{\updefault}{\color[rgb]{0,0,0}$y_{k_1(k_2+1)+2}$}%
}}}}
\put(1150,-3270){\makebox(0,0)[lb]{\smash{{\SetFigFont{8}{9.6}{\rmdefault}{\mddefault}{\updefault}{\color[rgb]{0,0,0}$x_1$}%
}}}}
\put(131,-2535){\makebox(0,0)[lb]{\smash{{\SetFigFont{8}{9.6}{\rmdefault}{\mddefault}{\updefault}{\color[rgb]{0,0,0}$y_{k_1(k_2+1)+1}$}%
}}}}
\put(151,-1779){\makebox(0,0)[lb]{\smash{{\SetFigFont{8}{9.6}{\rmdefault}{\mddefault}{\updefault}{\color[rgb]{0,0,0}$y_{k_1k_2+2}$}%
}}}}
\put(136,-3024){\makebox(0,0)[lb]{\smash{{\SetFigFont{8}{9.6}{\rmdefault}{\mddefault}{\updefault}{\color[rgb]{0,0,0}$y_2$}%
}}}}
\put(126,-3645){\makebox(0,0)[lb]{\smash{{\SetFigFont{8}{9.6}{\rmdefault}{\mddefault}{\updefault}{\color[rgb]{0,0,0}$y_{k_1+1}$}%
}}}}
\put(141,-3944){\makebox(0,0)[lb]{\smash{{\SetFigFont{8}{9.6}{\rmdefault}{\mddefault}{\updefault}{\color[rgb]{0,0,0}$y_{k_1+2}$}%
}}}}
\end{picture}%

 \end{center}
 
   Hence $Y$ is extremal and the theorem follows.           
               \end{proof}

               We conclude this note with a few more examples of arrays that serve as counterexamples to possible  generalizations of the results mentioned.
 
\begin{remark} 
               
               The array 
\[
F= \begin{pmatrix}
3 & 0 & 0 & 1\\
0&2&2&0\\
0&1&1&2\\
\end{pmatrix}               
\]       
is an element of $\M(3,4)$ with   $ \Abs{\supp F}= 7=S(3,4)+1$ but it is not extremal. Hence it is not the case that any 
doubly stochastic  array in $\M(n,m)$  whose support is just one  above the minimum support of $\M(n,m)$ must be extremal.
\end{remark}

\begin{remark}
 
 Clearly a possible generalization of Birkhoff's theorem  to non-square doubly stochastic arrays, stating  that any two extremal arrays in $\M(n,m)$  are equivalent fails. This can be easily seen as there exist plenty of examples of extremal arrays $A,B \in \M(n,m)$  with   $ \Abs{\supp A} \neq \Abs{\supp B}$. As minimum and extremal arrays coincide in $\M(n,n)$, we entertained  the idea that, maybe, any two minimum arrays in $\M(n,m)$ are equivalent. (If this were true  Birkhoff's theorem would be a special case.) But this fails too, as the next two minimum arrays in $\M(4,5)$ prove.
 \[
\begin{pmatrix}
1 &4 &0 &0 &0\\
0 & 0 &4 &0 &1 \\
3 &0 & 0& 2 &0\\
0& 0&0 &2&3
\end{pmatrix}
\text{ and } 
\begin{pmatrix}
4 &0 & 0  &0 & 1\\
0 & 4 & 0 &0& 1\\
0 & 0& 4 &0&1\\
0&0&0&4&1
\end{pmatrix}
\]     
\end{remark}

Nevertheless, we have not managed, so far, to produce two minimum arrays whose set of entries are equal (counting multiplicities)  without being  equivalent.  We should mention here that the way $\F(n,m)$ are constructed ensures that the entries of $\F(n,m)$ are $\{ n, r_1, r_2, \cdots r_t\}$
(using the notation in \eqref{eq1}) appearing  with multiplicities 
$$\{ k_1n,\, k_2r_1,\, k_3r_2,\, \cdots , k_{t+1}r_t \}$$ respectively.             

\section*{Acknowledgement}
The author would like to thank Ian Wanless for his valuable remarks  about the presentation of the paper and 
M. Etkind and N. Lev for spotting a gap in a previous version of Theorem 1.

\end{document}